\documentclass[12pt]{article}

\usepackage[a4paper]{geometry}
\usepackage{caption2}
\usepackage{multirow}
\usepackage{enumitem}
\usepackage{mathrsfs}
\usepackage{makecell}
\usepackage{algorithm,algorithmic}
\usepackage{amsmath}
\usepackage{amssymb}
\usepackage{amsfonts}
\usepackage{graphicx}
\usepackage{epstopdf}
\usepackage{float}
\usepackage{amsthm}
\usepackage[unicode={ture},colorlinks,
            linkcolor=black,
            anchorcolor=blue,
            citecolor=green]{hyperref}
\usepackage{appendix}
\allowdisplaybreaks

\makeatletter
\newenvironment{breakablealgorithm}
  {
   \begin{center}
     \refstepcounter{algorithm}
     \hrule height.8pt depth0pt \kern2pt
     \renewcommand{\caption}[2][\relax]{
       {\raggedright\textbf{\ALG@name~\thealgorithm} ##2\par}%
       \ifx\relax##1\relax 
         \addcontentsline{loa}{algorithm}{\protect\numberline{\thealgorithm}##2}%
       \else 
         \addcontentsline{loa}{algorithm}{\protect\numberline{\thealgorithm}##1}%
       \fi
       \kern2pt\hrule\kern2pt
     }
  }{
     \kern2pt\hrule\relax
   \end{center}
  }
\makeatother

\def\ma{\mathscr{A}}
\def\mb{\mathscr{B}}

\def\mf{\mathscr{F}}
\def\mg{\mathscr{G}}
\def\mh{\mathscr{H}}

\def\mj{\mathscr{J}}

\def\mp{\mathscr{P}}
\def\mr{\mathscr{R}}

\def\le{\leqslant}
\def\ge{\geqslant}
\def\bs{\setminus}

\def\gg{g_{t,p}(n_1,\dots,n_p;k_1,\dots,k_p)}

\def\gr{g_{t,p}(n_1,\dots,n_p;r_1,\dots,r_p)}
\def\fr{\mf_{\boldsymbol{r}}}

\def\fs{\mf_{\boldsymbol{s}}}

\def\dr{\boldsymbol{r}}
\def\ds{\boldsymbol{s}}

\newtheorem{thm}{Theorem}[section]
\newtheorem{lem}[thm]{Lemma}

\newtheorem{cx}[thm]{Conjecture}

\title{Extremal $t$-intersecting families for direct products}
\author{
	Tian Yao \quad Benjian Lv \quad Kaishun Wang\footnote{Corresponding author.  E-mail address: yaotian@mail.bnu.edu.cn(T. Yao),\newline bjlv@bnu.edu.cn(B. Lv),\ wangks@bnu.edu.cn(K. Wang)} \\
	{\footnotesize    \em  Sch. Math. Sci. {\rm \&} Lab. Math. Com. Sys.,
		Beijing Normal University, Beijing, 100875,  China}
}
\date{}

\begin{document}
\maketitle
\begin{abstract}

In this paper, by shifting technique we study $t$-intersecting families for direct products where the ground set is divided into several parts. Assuming the size of each part is sufficiently large, we determine all extremal $t$-intersecting families for direct products. We also prove that every largest $t$-intersecting subfamily of a more general family introduced by Katona is trivial under certain conditions.


\end{abstract}
\section{Introduction}

Let $n$ and $k$ be two integers with $0\le k\le n$. For an $n$-element set $X$, denote the set of all subsets and the collection of all $k$-subsets of $X$ by $2^X$ and $\binom{X}{k}$, respectively. Given a positive integer $t$, we say a family $\mf\subset2^{X}$ is \emph{$t$-intersecting} if $|A\cap B|\ge t$ for any $A,B\in\mf$. A $t$-intersecting family is called \emph{trivial} if every element of this family contains a fixed $t$-subset of $X$. When $t=1$, we usually omit $t$. The famous Erd\H{o}s-Ko-Rado theorem \cite{EKR} states that if $\mf\subset\binom{X}{k}$ is $t$-intersecting and $n>n_0(k,t)$,
then
$$|\mf|\le\binom{n-t}{k-t},$$
and the equality holds if and only if $\mf=\left\{F\in\binom{X}{k}:T\subset F\right\}$ for some $T\in\binom{X}{t}$.

It is well-known that the smallest value of $n_0(k,t)$ is $(t+1)(k-t+1)$, which was proved by Frankl \cite{Fn} for $t\ge15$, and confirmed by Wilson \cite{Wn} for all $t$ via the eigenvalue method. In \cite{Fn}, Frankl also put forward a conjecture about the maximum size of a $t$-intersecting subfamily of $\binom{X}{k}$ for $n>2k-t$. This conjecture was proved by Ahlswede and Khachatrian \cite{CEKR}.

The Erd\H{o}s-Ko-Rado theorem has been extended to different mathematical objects, such as vector spaces \cite{VEKR,CVEKR}, attenuated spaces \cite{AEKR2}, 
permutation groups \cite{PEKR3}, 
$2$-transitive groups \cite{2EKR1}, 
labeled sets \cite{LSEKR2} and partition sets \cite{PSEKR1}.

In \cite{DEKR}, Frankl studied intersecting families for direct products. For convenience, set $X=[n]:=\{1,\dots, n\}$ in the following. Let $p,n_1,\dots,n_p$ be positive integers such that $n=n_1+\cdots+n_p$. Then $X$ can be partitioned into $p$ parts $X_1,X_2,\dots,X_p$ where
$$X_1=[n_1],\ X_i=\left[\sum_{j\le i}n_j\right]\setminus\left[\sum_{j\le i-1}n_j\right],\quad i=2,\dots,p.$$ 
For positive integers $k_i\in [n_i]$ with $k=k_1+\cdots+k_p$, write
$$\mh_1:=\binom{X_1,\dots,X_p}{k_1,\dots,k_p}=\left\{F\in\binom{X}{k}:|F\cap X_i|=k_i,\ i=1,\dots,p\right\}.$$
Observe that $|\mh_1|=\prod_{j\in[p]}\binom{n_j}{k_j}$. For each $x\in X_l$, the size of $\{A\in\mh_1:x\in A\}$ is ${k_l}|\mh_1|/{n_l}$.
Frankl gave the maximum size of an intersecting subfamily of $\mh_1$ by the eigenvalue method.
\begin{thm}\rm{(\cite{DEKR})}\label{F}
	Suppose $\mf\subset\mh_1$ is an intersecting family and $n_i\ge2k_i$ for $i=1,\dots,p$. Then
	$$\dfrac{|\mf|}{|\mh_1|}\le\max_{i\in[p]}\dfrac{k_i}{n_i}.$$
\end{thm}

Recently, Kwan et al. \cite{DHM} determined the maximum size of a non-trivially intersecting subfamily of $\mh_1$ when $n_1,\dots,n_p$ are sufficiently large and so disproved a conjecture of Alon and Katona, which was also mentioned in \cite{MDEKR}.
The maximum sum of sizes of cross intersecting subfamilies of $\mh_1$ was determined by Kong et al. \cite{CDSUM}. Ahlswede et al. \cite{LDEKR} completely determined the maximum size of a $(t_1,\dots,t_p)$-intersecting subfamily of $\mh_1$, in which any two sets intersect in at least $t_i$ elements of $X_i$ for some $i\in[p]$. 

In this paper, we study $t$-intersecting subfamilies of $\mh_1$. One of our main results is the following.
\begin{thm}\label{T1}
	Suppose $\mf\subset\mh_1$ is a $t$-intersecting family. If $n_i>2(t+1)pk_i^2$ for any $i\in[p]$, then
	\begin{equation*}\label{EE}
	|\mf|\le\max_{\substack{t_1+\cdots+t_p=t\\ t_1,\dots,t_p\in\mathbb{N}}}\prod_{i\in[p]}\binom{n_i-t_i}{k_i-t_i}.
	\end{equation*}
	Moreover,
	the equality holds if and only if
	$$\mf=\{F\in\mh_1: T\subset F\},$$
	where $T\in\binom{X}{t}$ such that
		\begin{equation}\label{ES2}
	\dfrac{k_i-|T\cap X_i|}{n_i-|T\cap X_i|}\le\dfrac{k_j-|T\cap X_j|+1}{n_j-|T\cap X_j|+1}
\end{equation}
	for any $i\in[p]$ whenever $|T\cap X_j|\ge1$.
 \end{thm}

We remark here that $t$-intersecting subfamilies of $\mh_1$ with maximum size may not be trivial when $n_1,\dots,n_p$ are small. Under the condition that $p=t=2$, $n_1=8$, $n_2=10$ and $k_1=k_2=4$, it is routine to check that the $2$-intersecting family $\left\{A\in\mh_1: |A\cap[4]|\ge3\right\}$ has a larger size than the largest trivially $2$-intersecting subfamily of $\mh_1$. 

In \cite{MDEKR}, Katona extended $\mh_1$ to a more general case. For a non-empty finite set $\mr\subset\underbrace{\mathbb{Z}^+\times\cdots\times\mathbb{Z}^+}_{p}$, write
$$\mh_2:=\bigcup_{(r_1,\dots,r_p)\in\mr}\binom{X_1,\dots,X_p}{r_1,\dots,r_p}.$$
For convenience, let $b$ and $c$ denote the maximum and minimum of numbers appearing in some elements of $\mr$, respectively. By the cyclic method, Katona proved the following result.
\begin{thm}\rm{(\cite{MDEKR})}
	Suppose $p=2$ and $n_1,n_2\ge9b^2$. If $\mf\subset\mh_2$ is intersecting, then $|\mf|$ cannot exceed the size of the largest trivially intersecting subfamily of $\mh_2$.
\end{thm}

Our another main result extends Katona's result.
 \begin{thm}\label{T2}
	Suppose  $t\le c$. If $n_i>2(t+1)pb^{t+2}$ for any $i\in[p]$, then every largest $t$-intersecting subfamily of $\mh_2$ is trivial.
\end{thm}

Write
$$\mh_3:=\left\{F\in\binom{X}{k}: |F\cap X_i|\ge a_i,\ i=1,\dots,p\right\},$$
where $a_1,\dots,a_p$ are integers with $a_1+\cdots+a_p\le k$ and $0\le a_i<n_i$. In \cite{CON}, Frankl et al. put forward the following conjecture.
\begin{cx}\rm{(\cite{CON})}\label{cx}
	If $n_i\ge2a_i$ for all $i$ and $n_i>k-\sum_{j=1}^pa_j+a_i$ for all but at most one $i\in[p]$ such that $a_i>0$, then the largest intersecting subfamily of $\mh_3$ is trivial.
\end{cx}

As a corollary of Theorem \ref{T2}, Conjecture \ref{cx} is true when $a_1,\dots,a_p$ are positive and each $X_i$ has a size larger than $4p(k-\sum_{i=1}^pa_i+\max_{i\in[p]}a_i)^3$.

In Section 2, we will focus on the shifting technique and prove some useful results for direct products.
In Section 3, we will give the proof of our main results.  
\section{Shifting technique for direct products}

In this section, we investigate the shifting technique and prove some useful results for direct products.

For any $i,j\in X$ and $F\subset X$,
define
\[\delta_{i,j}(F)=
\begin{cases}
(F\setminus\{j\})\cup\{i\},&j\in F, i\not\in F;\\
F,&otherwise.
\end{cases}
\]
Let $\Delta_{i,j}$ be the operation on a family $\mf\subset2^X$ defined by
$$\Delta_{i,j}(\mf)=\{\delta_{i,j}(F):F\in\mf\}\cup\{F\in\mf: \delta_{i,j}(F)\in\mf\}.$$
We have $|\Delta_{i,j}(\mf)|=|\mf|$.

A family $\mf\subset2^X$ is called \emph{shifted} if $\Delta_{i,j}(\mf)=\mf$ holds for any $i,j\in X$ with $i<j$. By applying such operations repeatedly to a subfamily of $2^X$ we can get a shifted family.

We say two non-empty subfamilies $\ma$ and $\mb$ of $2^X$ are \emph{cross $t$-intersecting} if $|A\cap B|\ge t$ for any $A\in\ma$ and $B\in\mb$. The following lemma states that the shifting operation keeps such intersection property.
\begin{lem}\rm{(\cite[Lemma 2.1]{CP})}\label{L1}
	Let $\ma$ and $\mb\subset2^X$ be cross $t$-intersecting families.
	\begin{itemize}
		\item[\rm{(\romannumeral1)}] For any $i,j\in X$,
		$\Delta_{i,j}(\ma)$ and $\Delta_{i,j}(\mb)$ are still cross $t$-intersecting.
		\item[\rm{(\romannumeral2)}] If $t\le r\le s\le n$, $\ma\subset\binom{X}{r}$, $\mb\subset\binom{X}{s}$, and $\ma$ and $\mb$ are shifted, then $|A\cap B\cap[r+s-t]|\ge t$ for any $A\in\ma$ and $B\in\mb$.
	\end{itemize}
\end{lem}

For $\mf\subset\mh_2$,
if $\Delta_{i,j}(\mf)=\mf$ holds for any $i,j\in X_l$ with $i<j$,
we say $\mf$ is $l$-\emph{shifted}. Similar to the single-part case, one gains an $l$-shifted family by doing the shifting operation repeatedly on $\mf$. Notice that Lemma \ref{L1}(\romannumeral1) still holds for $\ma\subset\binom{X_1,\dots,X_p}{r_1,\dots,r_p}$ and $\mb\subset\binom{X_1,\dots,X_p}{s_1,\dots,s_p}$. 

For $l\in[p]$ and a positive integer $s\le n_l$, denote the collection of the first $s$ elements of $X_l$ by $Q_l(s)$. The next lemma is an extension of Lemma \ref{L1}(\romannumeral2).
\begin{lem}\label{L2}
	Suppose $n_i>r_i+s_i-1$ for any $i\in[p]$.
	Let $\ma\subset\binom{X_1,\dots,X_p}{r_1,\dots,r_p}$ 
	and
	$\mb\subset\binom{X_1,\dots,X_p}{s_1,\dots,s_p}$ be cross $t$-intersecting families. If $\ma$ and $\mb$ are $l$-shifted for any $l\in[p]$, then
	$$\sum_{i=1}^p|A\cap B\cap Q_i(r_i+s_i-1)|\ge t$$
	for any $A\in\ma$ and $B\in\mb$.
\end{lem}
\begin{proof}
	For each $i\in[p]$, write
	$$D_i:=Q_i(r_i+s_i-1)\setminus(A\cup B),\ E_i:=(A\cap B\cap X_i)\setminus Q_i(r_i+s_i-1).$$
	Note that 
	\begin{equation}\label{221}
	r_i+s_i
	=|A\cap X_i|+|B\cap X_i|\ge 2|E_i|+|(A\cup B)\cap Q_i(r_i+s_i-1)|,
	\end{equation} 
	\begin{equation}\label{222}
	|D_i|=r_i+s_i-1-|(A\cup B)\cap Q_i(r_i+s_i-1)|.
	\end{equation}
	If $|E_i|\neq\emptyset$, then $|D_i|\ge|E_i|$ from (\ref{221}) and (\ref{222}). 
	
	Let $G_i$ be an $|E_i|$-subset of $D_i$. Write
	$$C:=\left(B\bs
	\bigcup_{i\in[p]}E_i\right)\cup\left(\bigcup_{i\in[p]}G_i\right).$$
	Observe that, for each $i\in[p]$,
	\begin{equation*}
	\begin{aligned}
	C\cap A\cap X_i&=\left((B\setminus E_i)\cup G_i\right)\cap A\cap X_i=A\cap B\cap Q_i(r_i+s_i-1).
	\end{aligned}
	\end{equation*}
	When $E_i\neq\emptyset$, notice that $\max G_i<\min E_i$ and $|E_i|=|G_i|$. Thus $C$ can be obtained by doing a series of shifting operations on $B$. Since $\mb$ is $l$-shifted for any $l\in[p]$, we have $C\in\mb$. So $|A\cap C|\ge t$. Hence
	$$\sum_{i=1}^p|A\cap B\cap Q_i(r_i+s_i-1)|=\sum_{i=1}^p|A\cap C\cap X_i|=|A\cap C|\ge t,$$
	as desired.
\end{proof}

Given positive integers $g,h$ with $g\ge2h$, it is well-known that the \emph{Kneser graph} $KG(g,h)$ is the graph on the vertex set $\binom{[g]}{h}$, with an edge between two vertices if and only if they are disjoint. To characterize extremal structures in Theorems \ref{T1} and \ref{T2}, we need a property of Kneser graphs which is derived from Theorem 1 in \cite{CONNECTED}.
\begin{lem}\label{KCON}
	For Kneser graphs $KG(g_1,h_1),\dots,KG(g_w,h_w)$ with $g_i>2h_i$ for any $i\in[w]$, their direct product $\prod_{i\in[w]}KG(g_i,h_i)$ is connected.
\end{lem}

For $\mh\subset2^X$, we say $\mf\subset\mh$ is a \emph{full $t$-star} in $\mh$ if $\mf$ is the collection of all sets in $\mh$ containing a fixed $t$-subset of $X$. For each $i\in[p]$, let $b_i$ be the maximum number appearing in the $i$-th coordinate of some elements of $\mr$.
\begin{lem}\label{L3}
	Let $\mf\subset\mh_2$ be a $t$-intersecting family. Suppose $n_m>2(t+1)b_m$ for any $m\in[p]$. For $l\in[p]$ and $i,j\in X_l$, if $\Delta_{i,j}(\mf)$ is a full $t$-star in $\mh_2$, then $\mf$ is also a full $t$-star in $\mh_2$.
\end{lem}
\begin{proof}
	For $\dr=(r_1,\dots,r_p)\in\mr$, let $\fr$ denote $\mf\cap\binom{X_1,\dots,X_p}{r_1,\dots,r_p}$ in the rest of the paper. Write
	$$\fr(l):=\{F\bs X_l: F\in\fr\}.$$
	For each $R\in\fr(l)$, let
	$$\mg_R:=\left\{R'\in\binom{X_l}{r_l}:R\cup R'\in\fr\right\}.$$
 Observe that
	\begin{equation}\label{231}
	\fr=\bigcup_{R\in\fr(l)}\{R\cup R': R'\in\mg_R\},\ \Delta_{i,j}(\fr)=\bigcup_{R\in\fr(l)}\{R\cup R'': R''\in\Delta_{i,j}(\mg_R)\}.
	\end{equation}
	By assumption, there exists $T_0\in\binom{X}{t}$ such that $\Delta_{i,j}(\mf)=\{F\in\mh_2: T_0\subset F\}$, which implies that \begin{equation}\label{fr}
	\Delta_{i,j}(\fr)=\left\{F\in\binom{X_1,\dots,X_p}{r_1,\dots,r_p}: T_0\subset F\right\}.
	\end{equation} 
	We have $|\mg_R|=|\Delta_{i,j}(\mg_R)|=\binom{n_l-t_l}{r_l-t_l}$, where $t_l:=|T_0\cap X_l|$.
	
	If $T_0\cap X_l=\emptyset$, we get $\mg_R=\Delta_{i,j}(\mg_R)$ from $\Delta_{i,j}(\mg_R)=\binom{X_l}{r_l}$. By (\ref{231}), $\fr=\Delta_{i,j}(\fr)$. Hence $\mf=\Delta_{i,j}(\mf)$, as desired.
	
	Now suppose $T_0\cap X_l\neq\emptyset$. By (\ref{fr}), we have 
	$$\fr(l)=\{G\subset X\bs X_l: T_0\bs X_l\subset G,\ |G\cap X_m|=r_m,\ m\in[p]\bs\{l\}\}.$$
	Note that $n_m>2(t+1)r_m$ for any $m\in[p]$. Then given $R_0\in\fr(l)$, there exists $S_0\in\fr(l)$ such that $R_0\cap S_0=T_0\bs X_l$. Since $\fr$ is $t$-intersecting, $\mg_{R_0}$ and $\mg_{S_0}$ are cross $t_l$-intersecting families with $|\mg_{R_0}||\mg_{S_0}|=\binom{n_l-t_l}{r_l-t_l}^2$. By Theorem 1 in \cite{CROSS}, we get
	$$\mg_{R_0}=\mg_{S_0}=\left\{G\in\binom{X_l}{r_l}: T'_l\subset G\right\}$$ for some $T'_l\in\binom{X_l}{t_l}$. Next we prove $\mg_S=\mg_{R_0}$ for any $S\in\fr(l)\bs\{R_0\}$.
	
	For each $S\in\fr(l)$, we have $|(S\setminus T_0)\cap X_m|=r_m-t_m$, $m\in[p]\setminus\{l\}$. Thus the set $\{R\setminus T_0:R\in\fr(l)\}$ can be seen as the vertex set of the graph $\prod_{m\in[p]\setminus\{l\}}KG(n_m-t_m,r_m-t_m)$. Notice that $n_m-t_m>2(r_m-t_m)$. Suppose $S\neq R_0$. By Lemma \ref{KCON}, this graph contains a walk
	$$R_0\setminus T_0,A_1,\dots,A_{z}=S\setminus T_0.$$
	Let $B_0=R_0$, $B_1=A_1\cup(T_0\setminus X_l)$, $\dots$, $B_z=S\in\fr(l)$. Then $B_{q}\cap B_{q+1}=T_0\setminus X_l$ for $q=0,1,\dots,z-1$. 	Consequently $\mg_{R_0}=\mg_{B_1}=\dots=\mg_S$. 
	
	For any $R\in\fr(l)$, $\mg_R$ is the collection of all $r_l$-subsets of $X_l$ containing $T_l'$. Hence
	\begin{equation}\label{jg}\begin{aligned}\fr&=\left\{R\cup R': R\in\fr(l),\ T_l'\subset R'\in\binom{X_l}{r_l}\right\}\\
	&=\left\{F\in\binom{X_1,\dots,X_p}{r_1,\dots,r_p}: T_1\subset F\right\},\end{aligned}\end{equation}
	where $T_1:=(T_0\setminus X_l)\cup T_l'$.
	
	For $\ds=(s_1,\dots,s_p)\in\mr$, by (\ref{jg}), there exists $T_2\in\binom{X}{t}$ such that $\fs$ is the collection of all sets in $\binom{X_1,\dots,X_p}{s_1,\dots,s_p}$ containing $T_2$. Since $n_m>2(t+1)b_m$ for any $m\in[p]$, there are $F_1\in\fr$ and $F_2\in\fs$ such that $(F_1\bs T_1)\cap(F_2\bs T_2)=\emptyset$. Then $t\le|F_1\cap F_2|=|T_1\cap T_2|\le t$, which implies that $T_1=T_2$. Thus for any $\ds\in\mr$, $\fs$ is the collection of all sets in $\binom{X_1,\dots,X_p}{s_1,\dots,s_p}$ containing $T_1$, which implies that the desired result follows.
\end{proof}
\section{Proof of main results}

In this section, we shall prove our main results.

Let $\mf\subset\mh_2$ be a $t$-intersecting family. If $\mf=\emptyset$, there is nothing to prove. So suppose that $\mf\neq\emptyset$. Besides, according to Lemma \ref{L3}, we may assume that $\mf$ is $l$-shifted for any $l\in[p]$. 

Recall that $b_i=\max\limits_{(r_1,\dots,r_p)\in\mr}r_i$ for $i=1,\dots,p$. Write
$$K:=\bigcup_{i=1}^pQ_i(2b_i-1),\ \alpha(\mf):=\min_{F\in\mf}|F\cap K|.$$ 
We have $\alpha(\mf)\ge t$. Indeed, since two non-empty subfamilies $\fr$ and $\fs$ are cross $t$-intersecting and $l$-shifted for any $l\in[p]$, by Lemma \ref{L2} we get
\begin{equation}\label{L3E}
|F\cap K|\ge\sum_{i=1}^p|F\cap G\cap Q_i(2b_i-1)|\ge t,
\end{equation}
where $F\in\fr$ and $G\in\fs$.

\begin{lem}\label{t}
	Suppose $\mf\subset\mh_2$ is a $t$-intersecting family. If $\alpha(\mf)=t$ and $\mf$ is $l$-shifted for any $l\in[p]$, then 
	\begin{equation}\label{TR}
	\begin{aligned}
	|\mf|\le\max_{\substack{t_1+\cdots+t_p=t\\ t_1,\dots,t_p\in\mathbb{N}}}\sum_{(r_1,\dots,r_p)\in\mr}\prod_{i\in[p]}\binom{n_i-t_i}{r_i-t_i}.
	\end{aligned}
	\end{equation}
	Moreover, when the equality holds, $\mf$ is a full $t$-star in $\mh_2$.
\end{lem}
\begin{proof}
	By assumption, there exists $F_0\in\mf$ such that $|F_0\cap K|=t$. By (\ref{L3E}), for any $G\in\mf$, we have
	\begin{equation}\label{S}
	F_0\cap K=\bigcup_{i\in[p]}\left(F_0\cap G\cap Q_i(2b_i-1)\right)\subset G.
	\end{equation}  
	Therefore, for any $\dr=(r_1,\dots,r_p)\in\mr$,
	$$|\fr|\le\prod_{i\in[p]}\binom{n_i-|F_0\cap Q_i(2b_i-1)|}{r_i-|F_0\cap Q_i(2b_i-1)|}.$$
	Then (\ref{TR}) follows from $|\mf|=\sum_{\dr\in\mr}|\fr|$. 
	
	By (\ref{S}), $\mf$ is a collection of some sets in $\mh_2$ containing $F_0\cap K$. So when the equality in (\ref{TR}) holds, $\mf$ is a full $t$-star in $\mh_2$.
	\end{proof}

For positive integers $t,p,n_1,\dots,n_p,k_1,\dots,k_p$ with $n_i>k_i$ and $k_1+\cdots+k_p\ge t$, write
$$\gg=\max_{\substack{t_1+\cdots+t_p=t\\ t_1,\dots,t_p\in\mathbb{N}}}\prod_{i\in[p]}\binom{n_i-t_i}{k_i-t_i}.$$
	\begin{proof}[\bf Proof of Theorem \ref{T1}]
		Notice that $\mh_1$ is a special case of $\mh_2$. In view of Lemma \ref{t}, we show that
		$$|\mf|<\gg$$ 
	when  $\alpha(\mf)\ge t+1$.
		 For convenience, if there is no confusion, we replace $\alpha(\mf)$ with $\alpha$ in the following.  
		
		By assumption, there exists $A_0\in\mf$ such that $|A_0\cap K|=\alpha$. Then for $F\in\mf$, we have $|F\cap K|\ge\alpha$ and $|F\cap K\cap A_0|\ge t$ by (\ref{L3E}). Thus
		\begin{equation}\label{E5}
		\mf\subset\bigcup_{J\in\binom{K}{\alpha},\ |J\cap A_0|\ge t}\left\{F\in\mh_1: J\subset F\right\}.
		\end{equation}
		
		Let $N$ be the collection of all non-negative integer solutions of the equation $x_1+\dots+x_p=\alpha-t$. For each $H\in\binom{K\cap A_0}{t}$ and $\beta=(c_1,\dots,c_p)\in N$, let $\mj(H,\beta)$ be the set of all $J\in\binom{K}{\alpha}$ with $H\subset J$ and $|(J\bs H)\cap X_i|=c_i$. Denote the number of $F\in \mh_1$ containing at least one element of $\mj(H,\beta)$ by $f(H,\beta)$. For each $J\in\binom{K}{\alpha}$ satisfying $|J\cap A_0|\ge t$, observe that $J$ is an element of some $\mj(H,\beta)$. Then by (\ref{E5}), we have
		\begin{equation*}\label{FHB}
		|\mf|\le\sum_{H\in\binom{K\cap A_0}{t}}\sum_{\beta\in N}f(H,\beta).
		\end{equation*}
		Observe that 
		$$|\mj(H,\beta)|\le\prod_{i\in[p]}\binom{2k_i-1}{c_i}\le\prod_{i\in[p]}(2k_i)^{c_i}.$$
		Thus
			\begin{equation*}\label{FHB2}
			\dfrac{f(H,\beta)}{\gg}\le\dfrac{\left(\prod\limits_{i\in[p]}(2k_i)^{c_i}\right)\cdot\left(\prod\limits_{i\in[p]}\binom{n_i-|H\cap X_i|-c_i}{k_i-|H\cap X_i|-c_i}\right)}{\prod\limits_{i\in[p]}\binom{n_i-|H\cap X_i|}{k_i-|H\cap X_i|}}\le\prod_{i\in[p]}\left(\dfrac{2k_i^2}{n_i}\right)^{z_i},
	\end{equation*}
	where $(z_1,\dots,z_p)\in N$ such that  $$\prod\limits_{i\in[p]}\left(\dfrac{2k_i^2}{n_i}\right)^{z_i}=\max\limits_{(c_1,\dots,c_p)\in N}\prod\limits_{i\in[p]}\left(\dfrac{2k_i^2}{n_i}\right)^{c_i}.$$
	 Note that $|N|=\binom{\alpha-t+p-1}{p-1}$ and $$\binom{x}{y}=\prod_{i=y+1}^x(1+\dfrac{y}{i-y})\le(y+1)^{x-y}$$ for any positive integers $x,y$ with $x\ge y+1$. By above discussion, we obtain
		\begin{align*}
		\dfrac{|\mf|}{g_{t,p}(n_1,\dots,n_p;k_1,\dots,k_p)}&\le\binom{\alpha}{t}\binom{\alpha-t+p-1}{p-1}\cdot\prod_{i\in[p]}\left(\dfrac{2k_i^2}{n_i}\right)^{z_i}\\
		&\le\left((t+1)p\right)^{\alpha-t}\cdot\prod_{i\in[p]}\left(\dfrac{2k_i^2}{n_i}\right)^{z_i}\\
		&=\prod_{i\in[p]}\left(\dfrac{2(t+1)pk_i^2}{n_i}\right)^{z_i}.
		\end{align*}
		 Since $n_i>2(t+1)pk_i^2$ for any $i\in[p]$, we have $|\mf|<\gg$, as desired.

		 For each $S\in\binom{X}{t}$, write 
		 $$\mp(S):=\{(i,j(i))\in\mathbb{Z}^2: i\in[p],\ 0\le j(i)<|S\cap X_i|\}.$$
		 Observe that
		 \begin{equation}\label{ES}
		 	e(S):=\dfrac{\prod_{i\in[p]}\binom{n_i-|S\cap X_i|}{k_i-|S\cap X_i|}}{\prod_{i\in[p]}\binom{n_i}{k_i}}=\prod_{(i,j)\in\mp(S)}\dfrac{k_i-j}{n_i-j}.
		 	\end{equation} 
		Let $T$ be a $t$-subset of $X$. To finish the proof, it is sufficient to show that $e(T)=\max_{S\in\binom{X}{t}}e(S)$ if and only if (\ref{ES2}) holds for any $i\in[p]$ whenever $|T\cap X_j|\ge1$. 
		
		Suppose that (\ref{ES2}) holds for any $i\in[p]$ whenever $|X\cap T_j|\ge1$. For each $S\in\binom{X}{t}\bs\{T\}$, from $$\dfrac{k_i}{n_i}>\dfrac{k_i-1}{n_i-1}>\cdots>\dfrac{1}{n_i-k_i+1},$$ we get
	 \begin{equation}\label{ES1}
	\min_{(i,j)\in\mp(T)\bs\mp(S)}\dfrac{k_i-j}{n_i-j}\ge\max_{(i,j)\in\mp(S)\bs\mp(T)}\dfrac{k_i-j}{n_i-j}.
	\end{equation} 
		By (\ref{ES}) and (\ref{ES1}), we have $e(T)/e(S)\ge1$. On the other hand, suppose $e(T)=\max_{S\in\binom{X}{t}}e(S)$. For each $i,j$ with $|T\cap X_j|\ge1$, let $T':=(T\bs\{u\})\cup\{v\}\in\binom{X}{t}$, where $u\in T\cap X_j$ and $v\in X_i\bs T$. By (\ref{ES}), we have
		$$\frac{k_i-|T\cap X_i|}{n_i-|T\cap X_i|}=\dfrac{e(T')}{e(T)}\cdot\dfrac{k_j-|T\cap X_j|+1}{n_j-|T\cap X_j|+1}\le\dfrac{k_j-|T\cap X_j|+1}{n_j-|T\cap X_j|+1}.$$
	Hence the desired result holds.	
\end{proof}

	It is not intuitive to find $T\in\binom{X}{t}$ such that the size of $\{F\in\mh_1: T\subset F\}$ is $\gg$. Thus we extract an algorithm about how to find all $|T\cap X_i|$ from the proof of Theorem \ref{T1}.
	\medskip
		\begin{breakablealgorithm}
		\caption{} \label{SF}
		\begin{algorithmic}[1]
			\STATE {\bf Input}\qquad $t,p,k_1,\dots,k_p,n_1,\dots,n_p$
			\STATE Let $A$ be the collection of $\dfrac{k_i-j}{n_i-j}$ for all $i$, $j$ with $i\in[p]$, $j=0,\dots,k_i-1$
			\STATE Sort $A$ in decreasing order $a_1,a_2,\dots$
			\STATE Let $A(f)$ be the collection of $(i,j)$ satisfying $\dfrac{k_i-j}{n_i-j}=f$ for $f\in A$
			\STATE Put $i\leftarrow1$, $c\leftarrow0$, $k\leftarrow0$, $G\leftarrow\emptyset$
			\WHILE{$k<t$}
			\STATE $k\leftarrow k+|A(a_i)|$
			\IF{$k\le t$}
			\STATE $G\leftarrow G\cup A(a_i)$
			\ELSE
			\STATE $c\leftarrow|A(a_i)|-k+t$
			\STATE $H\leftarrow\binom{A(a_i)}{c}$
			\ENDIF
			\STATE $i\leftarrow i+1$
			\ENDWHILE
			\IF{$c=0$}
			\FOR{$t_m$}
			\STATE $t_m\leftarrow|\{(m,j): (m,j)\in G\}|$
			\ENDFOR
			\STATE {\bf Output}\qquad $t_1,\ldots,t_p$
			\ELSE
			\FOR{$L\in H$}
			\STATE $J\leftarrow G\cup L$
			\FOR{$t_m$}
			\STATE $t_m\leftarrow|\{(m,j): (m,j)\in J\}|$
			\ENDFOR
			\STATE {\bf Output}\qquad $t_1,\ldots,t_p$
			\ENDFOR
			\ENDIF
		\end{algorithmic}
	\end{breakablealgorithm}
\medskip
\begin{proof}[\bf Proof of Theorem \ref{T2}]
	In consideration of Lemma \ref{t}, it is sufficient to show that
	\begin{equation*}\label{<}
	|\mf|<\max_{\substack{t_1+\cdots+t_p=t\\ t_1,\dots,t_p\in\mathbb{N}}}\sum_{(r_1,\dots,r_p)\in\mr}\prod_{i\in[p]}\binom{n_i-t_i}{r_i-t_i}
	\end{equation*}
 when $\alpha(\mf)\ge t+1$. W.o.l.g., suppose that $n_1=\min_{i\in[p]}n_i$.
	
	We may assume that $\fr\neq\emptyset$ for some $\dr=(r_1,\dots,r_p)\in\mr$, otherwise there is nothing to prove. Observe that $\fr$ is $t$-intersecting and $\alpha(\fr)\ge\alpha(\mf)\ge t+1$. From the proof of Theorem \ref{T1}, we get
	\begin{equation}\label{bt1}\dfrac{|\fr|}{\gr}\le\prod_{i\in[p]}\left(\dfrac{2(t+1)pr_i^2}{n_i}\right)^{w_i}\le\prod_{i\in[p]}\left(\dfrac{2(t+1)pb_i^2}{n_i}\right)^{q_i}, 	\end{equation}
	where $w_1+\cdots+w_p=\alpha(\fr)-t$ and $q_1+\cdots+q_p=\alpha(\mf)-t$. Notice that there exist non-negative integers $d_1,\dots,d_p$ with $d_1+\cdots+d_p=t$ such that
	\begin{equation}\label{bt2}
	\begin{aligned}
	\dfrac{\gr}{\binom{n_1-t}{r_1-t}\cdot\prod_{i=2}^p\binom{n_i}{r_i}}&=\left( \prod_{i\in[p]}\left(\prod_{j=0}^{d_i-1}\dfrac{r_i-j}{n_i-j} \right)\right) \cdot\left( \prod_{j=0}^{t-1}\dfrac{n_1-j}{r_1-j}\right)\le b^t.
	\end{aligned}
	\end{equation}
	Combining (\ref{bt1}) and (\ref{bt2}), we derive
	\begin{align*}
	\dfrac{|\fr|}{\binom{n_1-t}{r_1-t}\cdot\prod_{i=2}^p\binom{n_i}{r_i}}&\le b^t\cdot\prod_{i\in[p]}\left(\dfrac{2(t+1)pb_i^2}{n_i}\right)^{q_i}\le\left(\dfrac{2(t+1)pb^{t+2}}{n_1}\right)^{\alpha-t}<1
	\end{align*}
	from $n_1>2(t+1)pb^{t+2}$. Therefore, $|\mf|$ is smaller than the number of sets in $\mh_2$ containing $[t]$, which implies that the desired result follows.
\end{proof}
\section*{Acknowledgement}

This research is supported by NSFC (11671043) and NSF of Hebei Province (A$2019205092$).

\end{document}